\documentclass[12pt,letterpaper]{amsart}
\usepackage{amsmath, amsthm, amssymb, xspace, mathrsfs}
\usepackage{geometry}
\usepackage[breaklinks=true]{hyperref}
\usepackage{amsmath,amscd}

\theoremstyle{plain}

\newtheorem{conjecture}{Conjecture}
\newtheorem{theorem}[equation]{Theorem}
\newtheorem{lemma}[equation]{Lemma}
\newtheorem{prop}[equation]{Proposition}
\newtheorem{cor}[equation]{Corollary}
\newtheorem{problem}{Problem}

\theoremstyle{definition}

\numberwithin{equation}{section}

\newcommand{\comment}[1]{}

\DeclareMathOperator{\spec}{\ensuremath{Spec}}

\DeclareMathOperator{\ad}{\ensuremath{ad}}

\DeclareMathOperator{\Hom}{\ensuremath{Hom}}
\DeclareMathOperator{\Fr}{\ensuremath{Fr}}

\title[the  isomorphism problem for the rings of differential operators]{On the isomorphism problem for the rings of differential operators on smooth affine varieties}

\author{Akaki Tikaradze}
\email{ Akaki.Tikaradze@utoledo.edu}
\address{University of Toledo, Department of Mathematics \& Statistics, 
Toledo, OH 43606, USA}

\begin{document}

\begin{abstract}
We show that given two smooth affine varieties over $\mathbb{C}$ such that their rings of differential operators
are Morita equivalent, then  corresponding cotangent bundles are isomorphic
as symplectic varieties.
 
\end{abstract}
\maketitle


\section{Introduction}

Throughout given a smooth affine scheme $X$ of finite type over a ring $S$, by $D(X)$ (respectively $T^*(X)$) we will denote the ring of $S$-relative
algebraic differential operators
on $X$ (the $S$-relative cotangent bundle of $X$). Recall that $D(X)$ asa ring is generated by regular functions $\mathcal{O}(X)$ and vector fields
($S$-linear derivations)
 $T^1_X=\text{Der}_S( \mathcal{O}(X), \mathcal{O}(X))$
subject to the following relations:
$$ \theta\cdot f-f\cdot\theta=\theta(f),\quad  \theta\cdot\theta_1-\theta_1\cdot\theta=[\theta, \theta_1]_{\text{Lie}}, \quad f\in\mathcal{O}(X),\quad \theta,\theta_1\in T^1_X.$$
 
 Recall that the isomorphism problem for the rings of differential operators is as follows. 
 
\begin{problem}  Suppose that $X, Y$ are smooth
affine varieties  over $\mathbb{C}$, such that $D(X)\cong D(Y)$ as $\mathbb{C}$-algebras. Then is it true that $X\cong Y$?
\end{problem}

It is well known that the answer is yes when $X, Y$ are algebraic curves  as proved by Makar-Limanov \cite{ML}, while not much is known
for higher dimensional varieties (see introduction in \cite{BW}).
Much more generally, it has been conjectured by Orlov that if any two smooth varieties $X, Y$ over $\mathbb{C}$ have equivalent
derived categories of coherent $D$-modules, then $X\cong Y.$ Orlov's conjecture has been established for Abelian varieties
by Arinkin (see discussion in \cite{F}). As for smooth affine varieties $X$ and $Y,$ a derived equivalence between $D(X)$ and $D(Y)$ is the same as
a (shifted) Morita equivalence by \cite{YZ}, it follows that Orlov's conjecture for smooth affine varieties boils down to determining whether
a Morita equivalence between $D(X)$ and $D(Y)$ for smooth affine varieties $X, Y$ implies that $X\cong Y.$
Such a result for smooth affine curves over $\mathbb{C}$ was indeed established by Berest and Wilson [\cite{BW}, Theorem 3.3].
Their proof relies on an explicit description of algebras Morita equivalent to $D(X)$ for a smooth affine curve $X$, as well as mad subalgebras approach of
Makar-Limanov. We also remark that the smoothness assumption for high dimensional varieties is essential, as examples constructed in \cite{LSS} showed.

In this note we show the following result.

\begin{theorem}\label{main}
Let $X, Y$ be smooth affine varieties over $\mathbb{C}$ such that $D(X)$ is derived
Morita equivalent to  $D(Y).$
Then $T^*(X)\cong T^*(Y)$ as symplectic varieties.
\end{theorem}

As an immediate corollary, we are able to reprove the aforementioned result of Berest and Wilson, see Corollary \ref{BW}.

In view of this result, it is tempting to make the following conjecture.

\begin{conjecture}
Let $X, Y$ be smooth affine varieties over $\mathbb{C}$ such that $T^*(X)\cong T^*(Y).$ 
 Then $X\cong Y.$
\end{conjecture}

Using  Theorem \ref{main}, the above conjecture would yield Orlov's conjecture for affine varieties.
Thus the conjecture allows us to transfer the isomorphism problem from the domain of noncommutative
ring theory to the one in symplectic algebraic geometry.

\section{Rings of differential operators in characteristic $p$}

In this section we briefly recall some well-known facts about rings of (crystalline) differential operators in characteristic $p$, which
play the crucial role in the proof of Theorem \ref{main}.
Throughout, given a ring, variety, module or a homomorphism
 over a ring $S,$ then its base change with respect to a homomorphism $S\to \bold{k}$ is denoted by the subscript $\bold{k}.$

Let $X$ be a smooth variety over a perfect field $\bold{k}$ of characteristic $p.$ Then it is a fundamental observation of Bezrukavnikov, Mirkovic and Rumynin
[\cite{BMR}, Theorem 2.2.3] that $D(X)$ (which in characteristic $p$ case is called the ring of crystalline differential operators) is an Azumaya algebra
over the Frobenius twist of the cotangent bundle on $X.$ In particaular, the center of $D(X)$ is canonically isomorphic to the Frobenius twist of
the ring of functions on $T^*(X).$ Also, recall that unlike characteristic 0 case, algebra $D(X)$ is not a subalgebra of $\Hom_{\bold{k}}(\mathcal{O}(X), \mathcal{O}(X)).$

Let $X$ be a smooth affine variety over a characteristic $p$ ring $S.$
Recall that the ring of functions on $T_S^*(X)$ is generated by $\mathcal{O}(X)$ and $T^1_X=\text{Der}_{S}(\mathcal{O}(X), \mathcal{O}(X))$
over $S.$ Also recall that a Frobenius twist of $X$ over $S$ is defined as $\Fr_S(X)=\spec(\mathcal{O}(X)^p\otimes_{S^p}S).$
Then the canonical isomorphism $$i_{S}:\Fr_S(\mathcal{O}(T^*(X)))\cong Z(D(X))$$ is defined as follows:
$$i_{S}(f)=f^p, i_{S}(\theta)=\theta^p-\theta^{[p]}, 
f\in \mathcal{O}(X), \theta\in T^1_X,$$
here $\theta^{[p]}\in T^1_X$ denoted the derivation, the $p$-th power of $\theta.$

Next, we need to discuss the definition and the computation of the reduction $\mod p$ Poisson bracket on $Z(D(X_{\bold{k}}))$ provided that the variety $X_{\bold{k}}$
was obtain by a base change from a variety $X$ defined over a finitely generated subring $S\subset\mathbb{C}.$

Recall that given an associative flat $\mathbb{Z}$-algebra $R$ and a prime number $p,$
 then the center $Z(R/pR)$ of its reduction $\mod p$  acquires a natural Poisson bracket (see for example [\cite{BK} Section 5.2]),
  which we refer to as the reduction $\mod p$ bracket,
defined as follows. Given $a, b\in Z(R/pR)$, let $z, w\in R$ be their respective lifts. 
Then the Poisson bracket $\lbrace a, b \rbrace$ is defined to be $$\frac{1}{p}[z, w] \mod p\in Z(R/pR).$$

Let $X$ be a smooth affine variety over a finitely generated $S\subset\mathbb{C}.$ Let $X_p$ denote its base
change $\mod p$ and $S_p=S/pS.$ So, $X_p$ is smooth over $S_p,$ and $D(X_p)=D(X)/pD(X).$ Thus, we may equip the center of $D(X_p)$ 
with the reduction $\mod p$ Poisson bracket.
Let $S\to\bold{k}$ be a base change to a perfect field of characteristic $p.$
 Noting that $Z(D(X_{\bold{k}}))=Z(D(X_p))\otimes_{S_p}\bold{k}$, we can extend the Poisson bracket from $Z(D(X_p))$ to
 a $\bold{k}$-linear Poisson bracket on $Z(D(X_{\bold{k}})).$
Then by a crucial computation in [\cite{BK}, Lemma 2]
the induced Poisson bracket on $\mathcal{O}(T^*(X_{\bold{k}}))$ via the isomorphism $i_{\bold{k}}:\mathcal{O}(T^*(X_{\bold{k}}))\to Z(D(X_{\bold{k}}))$
equals to the minus of the usual 
Poisson bracket on the cotangent bundle

\section{Auxiliary result} 

In this section we show a result which plays the crucial role in the proof of Theorem \ref{main}.
To state the following result, we recall that given a smooth affine variety $Y$ over a perfect field $\bold{k}$ of charactreristic $p$, we have the canonical
isomorphism $i_{\bold{k}}: \mathcal{O}(T^*(Y))\cong Z(D(Y)).$ In what follows, given a subset $V\subset A$ of an $R$-algebra and $n\in\mathbb{N}$, 
we denote by
$V^n$ the $R$-linear span of products of at most $n$ elements of $V.$

\begin{prop}\label{key}
Let $X$ be a smooth affine variety over a finitely generated ring $S\subset\mathbb{C}.$
Let $V\subset D(X)$ be a finite $S$-submodule, and $\theta\in T^1(X), 0\neq e\in D(X).$
Then there is a base change $S\to S'$ to a finitely generated subring of $\mathbb{C}$
together with finite $S'$-submodules $V'\subset \mathcal{O}(T^*(X_{S'})), W\subset D(X_{S'}),$ and $0\neq \phi\in\mathcal{O}(X_{S'}),$
such that the following holds. For any base change $S'\to\bold{k}$ to an algebraically closed field $\bold{k}$ of charactersistic $p$
we have 
$$\i_{\bold{k}}^{-1}(e^{-1}V_{\bold{k}}^p\cap Z(D(X_{\bold{k}})))\subset V'_{\bold{k}},\quad \theta_{\bold{k}}^p-\theta_{\bold{k}}^{[p]}\in \phi^{-p}W_{\bold{k}}^p.$$

\end{prop}

For the proof we utilize the following result from commutative algebra.
As usual, given a commutative ring $B$ of characteristic $p$, we denote by $\Fr(B)$ the image of $B$ under the Frobenius homomorphism.
Also, given a subset $V$ of an integral domain, We denote by $VV^{-1}$ the set of fractions of the form $vw^{-1}, v\in V, w\in V\setminus{0}.$

\begin{lemma}\label{key'}

Let $S\subset\mathbb{C}$ be a finitely generated subring, let $A$ be a finitely generated commutative ring over $S$, so that $A_{\mathbb{C}}=A\otimes_S\mathbb{C}$
is a domain. Let $V\subset A$ be a finite $S$-submodule, $0\neq f\in A$. Then there exists a 
base change $S\to S'$ to a finitely generated subring of $\mathbb{C}$ together with a
finite $S'$-submodule $W\subset A_{S'}$ such that the following holds.
For any base change $S'\to\bold{k}$ to an algebraically closed field of characteristic $p$, we have 
$$fV_{\bold{k}}^p\cap \Fr_p(A_{\bold{k}})\subset \Fr_p(W_{\bold{k}}), \quad V_{\bold{k}}{V'_{\bold{k}}}^{-1}\cap A_{\bold{k}}\subset W_{\bold{k}}.$$
\end{lemma}

\begin{proof}
We give a short geometric proof.
We may assume without loss of generality that $A$ is smooth over $S.$ Localizing $S$ if necessary,
there exists a smooth projective variety $X$ over $S$ containing $\spec(A)=U$ as an open subvariety.
Recall that given a rational function $g$ on $X,$ then div$(g)$ denotes the principal divisor of $g.$
Let $\lbrace f_1,\cdots, f_n\rbrace$ be a generating set of $V$ as an $S$-module. Let $D$ be an effective Weil divisor
on $X,$ such that $div(f_i)+D\geq 0$ for any $i.$
Let $L=\mathcal{O}(D)$ be the line bundle corresponding to $D.$ So, 
for any open subset $U\subset X$, a rational function $g$ belongs to $\Gamma(U, L)$ 
 if $(g)+D\geq 0.$ Now, localizing S if necessary, for any base change $S\to \bold{k}$ to an algebraically closed field of characteristic $p,$
  we still have that $div(f_i)+D_{\bold{k}}\geq 0.$
Let $g\in A_{\bold{k}}$ be so that $g^p\in V^p_{\bold{k}}.$
Then $p\text{div}(g)+pD_{\bold{k}}\geq 0.$ So, $g\in \Gamma(X_{\bold{k}}, L_{\bold{k}})=\Gamma(X, L)\otimes_S\bold{k}.$
Taking $W=\Gamma(X, L)\cap A,$ we are done since the global sections of a coherent sheaf of a projective variety over $S$ form
a finitely generated $S$-module.

Let $X\setminus U=\bigcup Z_i$, where $Z_i$ are closed codimension 1 irreducible subvarieties.
Let $D=\sum Z_i$ be an effective divisor. Let $m$ be large enough so that
$(g)+mD\geq 0$ for any $g\in V.$ 
Also, let $l$ be large enough so that any element of $V$ has at most order $l$ vanishing on each $Z_i.$
So, for $p\gg 0$ and a base change $S\to\bf{k}$ to an algebraically closed field of characteristic $p,$
the order of vanishing of any nonzero element of $V_{\bold{k}}$ on each $(Z_{i})_{\bold{k}}$ is at most $l.$
Then for any $f\in V_{\bold{k}}{V'_{\bold{k}}}^{-1}$ we have
$\text{div}(f)+(m+l)D_{\bold{k}}\geq 0.$ 
We put $L=\mathcal{O}((m+l)D).$
So, $f\in \Gamma(X_{\bold{k}}, L_{\bold{k}}).$ 
Taking $W=\Gamma(X, L)$ we are done just like above.

\end{proof}

\begin{proof}[Proof of Proposition \ref{key}]
At first, remark that if $U=X_f$ is a principal affine open subvariety $0\neq f\in\mathcal{O}(X)$
such that the assertions of the lemma hold for $U$, then they also hold for $X.$
Indeed,
we may view $D(X), \mathcal{O}(T^*(X))$ as subalgebras of $D(U), \mathcal{O}(T^*(U))$ respectively.
Similarly, $D(X_{\bold{k}})$ is a subalgebra of $D(X'_{\bold{k}})$ for any base change $S\to\bold{k}.$
let $V''\subset \mathcal{O}(T^*(U))$ be a finite $S$-submodule such that $V_{\bold{k}}^p\cap Z(D(U_{\bold{k}}))\subset i_{\bold{k}}(V''_{\bold{k}}).$
Then $V'=V''\cap \mathcal{O}(T^*(X))$ is a finite $S$-module and $V_{\bold{k}}^p\cap Z(D(X_{\bold{k}}))\subset i_{\bold{k}}(V'_{\bold{k}}).$
Let $\theta\in T^1_X\subset T^1_U.$ Then there is a finite $S$-submodule $W'\subset D(U)$ and $0\neq \phi_1\in\mathcal{O}(U)$ so that 
$\phi_1^p(\theta^p-\theta^{[p]})\in \phi_1^pW'^p_{\bold{k}}.$
Getting rid of the denomiantors, there exists $0\neq \phi\in\mathcal{O}(X)$ and $W\subset D(X)$ such that $\phi^p(\theta^p-\theta^{[p]})\in W_{\bold{k}}^p.$

Using the Noether normalization and replacing $X$ by its open subset if necessary,
we may assume that we have a finite etale map $\pi: X\to U,$ where 
 $U=\spec S[x_1,\cdots, x_n]_f$ and $0\neq f\in S[x_1,\cdots, x_n].$
Denote by $y_1,\cdots, y_n\in \text{Der}_S(\mathcal{O}(X))$ the pullbacks of vector fileds $\partial_{x_1},\cdots, \partial_{x_n}$
from $U$ to $X.$ So, localizing $X$ further if necessary, $y_1,\cdots, y_n$ form a basis of $\text{Der}_S(\mathcal{O}(X))$ as a free $\mathcal{O}(X)$-module.
Without loss of generality we may assume that $V=\bigoplus_{|I|<m} V_0y^I$ for some $m,$
 where $1\in V_0\subset \mathcal{O}(X)$
is a finite $S$-module. 
Let us choose a finite $S$-module $W\subset\mathcal{O}(X)$ containing $V_0$, such that $W$ generates $\mathcal{O}(X)$ over $S.$
Then for some $m$, we have 
$[y_i, W]\subset W^l, 1\leq i\leq n.$ Putting $M=W^l$, we get that $[y_i, M]\subset M^2, 1\leq i\leq n.$
 Then for any base change $S\to\bold{k}$ to a field of characteristic $p,$ we have
  $V_{\bold{k}}^p\subset \bigoplus_{|I|\leq pm}M_{\bold{k}}^py^{I}.$
Let $g\in \mathcal{O}(T^*(X_{\bold{k}}))=\mathcal{O}(X_{\bold{k}})[y_1,\cdots, y_n]$ so that
$ei_{\bold{k}}(g)\in V_{\bold{k}}^p.$ 
Write $e=\sum_{\alpha}e_{\alpha}y^{\alpha}, g=\sum_{\beta} g_{\beta}y^{\beta}.$ So $ei_{\bold{k}}(g)=\sum e_{\alpha}g_{\beta}^py^{\alpha+p\beta}.$
Now, choosing $p\gg 0$, we conclude that $e_{\alpha}g_{\beta}^p\in M_{\bold{k}}^p$ for all $\alpha, \beta.$
 So, by Lemma \ref{key'} there
exists a finite $S$-submodule $W'\subset \mathcal{O}(X)$ so that $g_{\beta}\in W_{\bold{k}}$ for all $\beta.$ 
Hence, by taking $W=\bigoplus_{\alpha\leq m}W'y^{\alpha}\subset \mathcal{O}(T^*(X))$,
we have that $g\in W_{\bold{k}}$ and we are done.

Finally, let $\theta\in T^1_X.$ So $\theta=\sum g_iy_i, g\in \mathcal{O}(X).$
Then 
$$\theta^p-\theta^{[p]}=i_{\bold{k}}(\theta)=\sum_ii_{\bold{k}}(g_i)i_{\bold{k}}(y_i)=\sum_ig_i^py_i^p.$$
Taking $W$ to be an $S$-span of $g_i, y_i,$ we have that $\theta^p-\theta^{[p]}\in W_{\bold{k}}.$

\end{proof}


\section{Proof of the main result}

The proof of Theorem \ref{main} is essentially a simple application of  the reduction $\mod p\gg 0$ method, following ideas in \cite{BK}.

At first, we recall the following standard result about lifting isomorphisms from characteristic $p\gg 0$ to characteristic 0.
 Its proof is included for the reader's convenience. 
\begin{lemma}\label{lift}

Let $S\subset\mathbb{C}$ be a finitely generated ring.
Let $A, B$ be Poisson $S$-algebras. Let $V, V'\subset A$ and $W,W'\subset B$ 
be finite $S$-submodules generating the respective algebra over $S.$
Suppose that for all $p\gg 0$ and base change $S\to \bold{k}$ to an algebraically closed field $\bold{k}$
of characteristic $p$, there exists a  pair of Poisson $\bold{k}$-algebra homomorphisms
$\phi_{\bold{k}}:A_{\bold{k}}\to B_{\bold{k}}, \psi_{\bold{k}}:B_{\bold{k}}\to A_{\bold{k}},$ such that $\phi_{\bold{k}}(V_{\bold{k}})\subset W_{\bold{k}}$
and $\psi_{\bold{k}}(W'_{\bold{k}})\subset V'_{\bold{k}}$ and $\psi_{\bold{k}}\phi_{\bold{k}}=Id_{A_{\bold{k}}}, \phi_{\bold{k}}\psi_{\bold{k}}=Id_{B_{\bold{k}}}.$
Then there exist Poisson $\mathbb{C}$-algebra homomorphisms $\phi:A_{\mathbb{C}}\to B_{\mathbb{C}}, \psi: B_{\mathbb{C}}\to A_{\mathbb{C}}$
so that
$\phi(V_{\mathbb{C}})\subset W_{\mathbb{C}},\psi(W'_{\mathbb{C}})\subset V'_{\mathbb{C}}$ and $\psi\phi=Id, \phi\psi=Id$ 
\end{lemma}
\begin{proof}

Localizing $S$ if necessary, we may assume that $V, V', W, W'$ are free $S$-modules, as are $A, B.$
Given an $S$-ring $R$, denote by $(R)$ the following set:
$$(R)=\lbrace (x, y), x\in \Hom_S(V_R, W_R), y\in \Hom_R(W'_R, V'_R)\rbrace$$ 
where $x$ (respectively $y$) extends to an $R$-Poisson algebra
homomorphism $\tilde{x}:A_{R}\to B_{R}$ (respectively $\tilde{y}:B_{R}\to A_{R})$ so that $\tilde{y}\tilde{x}=Id_{A_{R}}, xy=Id_{B_{R}}.$
It is clear that the defining conditions on $x, y$ can be restated in terms of vanishing of
certain polynomials over $S.$ Thus, the functor $R\to (R)$ is represented by an affine scheme $Z$ of finite type over $S,$
so $(R)=Z(R)=\Hom_S(\spec(R), X).$
Then by the assumption $Z(\bold{k})$ is nonempty for all base changes $S\to \bold{k}$ with char$(\bold{k})\gg 0.$ Hence
$Z(\mathbb{C})$ is also nonempty and we are done.

\end{proof}

In what follows, for a ring $A$ we denote by $A-\text{mod}$ the category of finitetly generated left $A$-modules.
Recall that two $S$-algebras $A, B$  are called Morita equivalent over $S$ if categories $A-\text{mod}$ and $B-\text{mod}$ are $S$-linearly equivalent.

\begin{proof}[Proof of Theorem \ref{main}]

There exists a finitely generated projective  left $D(X)$-module $P,$ such that $P$ is a generator for $D(X)-\text{mod}$ and
$D(Y)^{op}\cong \Hom_{D(X)}(P, P).$ Let $e: D(X)^n\to P$ be a $D(X)$-module projection for a suitable $n.$
We view $e$ as an idempotent in the matrix algebra $M_n(D(X)).$
Next, we recall that given an idempotent $e$ in a ring $A$, then the $A$-module $Ae$ is a generator if and only if $A=AeA.$
Indeed, as $\Hom_A(Ae, A)=eA,$ there exists a surjective homomorphism $A^n\to A$ (for some $n$) if and only if $1\in AeA.$ 
Thus, since $D(X)$ and $M_n(D(X))$ are Morita equivalent, $P$ being a generator of $D(X)-\text{mod}$ is equivalent to $M^n=M_n(D(X))e$
 being a generator for $M_n(D(X))-\text{mod}$, the latter is equivalent to the equality
$M_n(D(X))=M_n(D(X))eM_n(D(X)).$ We have a corresponding
 $\mathbb{C}$-algebra isomorphism
 $$\phi:D(Y)\cong eM_n(D(X))e\subset M_n(D(X)).$$ 
There exists $f_i, g_i\in M_n(D(X))$, so that $1=\sum_if_ieg_i.$
Let $S\subset\mathbb{C}$ be a large enough finitely generated ring over which $X, Y, e, f_i, g_i$, and $\phi$ are defined.
Enlarging $S$ further if necessary, let  $X_S, Y_S$ be smooth affine integral models of $X, Y$ of finite type over $S,$
such that $D(X_S), D(Y_S)$ are free $S$-modules,
 $e\in M_n(D(X_S)),$ and we have
an $S$-algebra homomorphism 
$$\phi_S: D(Y_S)\to eM_n(D(X_S))e$$
 which becomes an isomorphism over $\mathbb{C}.$
Since $1=\sum_if_ieg_i$ in $M_n(D(X_S))$, we have $M_n(D(X_S))=M_n(D(X_S))eM_n(D(X_S)).$ As $D(Y_S)$ is an $S$-torsion free module, $\phi_S$ is injective.
Now, let $x_1,\cdots, x_m$ be generators of $eM_n(D(X_S))e\subset M_n(D(X_S))$ as an $S$-algebra. Let $S'$ be a finitely generated $S$-algebra
such that $\phi^{-1}(x_i)\in D(Y_{S'})$ for all $i$. Hence, $\phi_{S'}$ is surjective and therefore is
an isomorphism.
In particular, $D(X_{S'}), D(Y_{S'})$ are Morita equivalent $S$-algebras.

Given a base change $S\to\bold{k}$ to
an algebraically closed field of characteristic $p\gg 0$, we get a $\bold{k}$-algebra isomorphism
$D(Y_{\bold{k}})^{op}\cong \Hom_{D(X_{\bold{k}})}(P_{\bold{k}}, P_{\bold{k}}),$
which in turn yields a $\bold{k}$-algebra isomorphism on the level of centers preserving the $\mod p$ reduction Poisson brackets
$$\phi_{\bold{k}}:Z(D(X_{\bold{k}}))\xrightarrow{\sim} Z(D(Y_{\bold{k}})).$$
Hence, we have an isomorphism of Poisson $\bold{k}$-algebras
$$\tilde{\phi}_{\bold{k}}=i^{-1}_{\bold{k}}\phi_{\bold{k}}i_{\bold{k}}:\mathcal{O}(T^*(X_{\bold{k}}))\xrightarrow{\sim}\mathcal{O}(T^*(Y_{\bold{k}})).$$

Now using Lemma \ref{lift}, we see that to finish  the proof of  Theorem \ref{main},
it suffices to show that given $ g\in\mathcal{O}(Y)$ and $\theta\in T^1_Y,$  there exists
a finite $S$-module $V\subset \mathcal{O}(T^*(X)),$ such that 
after a base change $S\to \bold{k}$ to an algebraically closed field $\bold{k}$ of characteristic $p\gg 0$, we have 
$$i_{\bold{k}}^{-1}(\phi_{\bold{k}}(g^p)), i_{\bold{k}}^{-1}(\phi_{\bold{k}}(\theta^p-\theta^{[p]}))\in V_{\bold{k}}.$$
Indeed, 
recall that we have an idempotent $e\in M_n(D(X))$ and an $S$-algebra isomorphism $\psi:D(Y)\cong eM_n(D(X))e.$
Hence, by taking a finite $S$-submodule $W\subset D(X)$ containing all entries of $e$ and $\psi(g),$
 we get that $\phi_{\bold{k}}(g^p)\in eW_{\bold{k}}^pe.$
Let $f\in D(X)$ be a nonzero entry of the matrix $e$.
Next, we observe that the preimage of $eW_{\bold{k}}^pe$ under the isomorphism $Z(D(X_{\bold{k}}))\to eZ(D(X_{\bold{k}}))$
belongs to $f^{-1}W_{\bold{k}}^p\cap Z(D(X_{\bold{k}})),$ the latter by Lemma \ref{key} is contained in $V_{\bold{k}}$
for some finite $S$-submodule $V\subset \mathcal{O}(T^*(X)).$ So, $i^{-1}_{\bold{k}}(\phi_{\bold{k}}(g^p))\in V_{\bold{k}}$ as desired.

Using Lemma \ref{key} again, we have that there exists a nonzero $h\in \mathcal{O}(Y)$ and a finite $S$-submodue $W\subset D(X)$
such that
$$\phi_{\bold{k}}(h^p)\phi_{\bold{k}}(\theta^p-\theta^{[p]})\in eW_{\bold{k}}^pe\cap eZ(D(X_{\bold{k}})).$$
As we have shown already, there exists a finite $S$-submodule $V,$ 
such that $i_{\bold{k}}^{-1}(\phi_{\bold{k}}(h^p))\in V_{\bold{k}}$,
and the preimage of $eW_{\bold{k}}^pe\cap eZ(D(X_{\bold{k}}))$ in $\mathcal{O}(T^*(X_{\bold{k}}))$ is a subspace of
$V_{\bold{k}}.$
Therefore
$$i_{\bold{k}}^{-1}(\phi_{\bold{k}}(\theta^p-\theta^{[p]}))\in V_{\bold{k}}V_{\bold{k}}^{-1}.$$
Once again using Lemma \ref{key}, we conclude that there exists a finite $S$-module $V\subset\mathcal{O}(T^*(X))$ so that
$i_{\bold{k}}^{-1}(\phi_{\bold{k}}(\theta^p-\theta^{[p]}))\subset V_{\bold{k}}.$

\end{proof}

Finally, we present a proof of the aforementioned result of Berest and Wilson. The upshot of our proof is that
unlike the one given in \cite{BW}, our approach does not rely on the description of the Morita equivalent algebras
of $D(X)$ for a smooth algebraic curve $X.$ For the proof, we need to recall the definition of mad subalgebras of Poisson (or noncommutative)
algebras: given a Poisson algebra $A$, its subalgebra $B$ is called a mad subalgebra if it is a maximal Abelian subalgebra (with regards to the Poisson bracket)
such that for every element $f\in B$, the corresponding Poisson bracket $\ad(f)=\lbrace f, -\rbrace :A\to A$ is locally nilpotent.

\begin{cor}\label{BW}

Let $X, Y$ be smooth affine algebraic curves over $\mathbb{C}$, such that $D(X)$ is Morita equivalent to $D(Y).$
Then $X\cong Y.$

\end{cor}

\begin{proof}
By Theorem \ref{main}, we have an isomorphism of symplectic varieties $\theta: T^*(X)\cong T^*(Y).$
The  original argument of Makar-Limanov about mad subalgebras of $D(Y)$ [\cite{ML}, Lemma 2]  adapted to
the Poisson algebra $\mathcal{O}(T^*(Y))$ implies that if $Y$ is not isomorphic to the affine line, then
$\mathcal{O}(Y)$ is the unique mad subalgebra of $\mathcal{O}(T^*(Y)).$ We may assume without loss of generality
that $Y\neq \mathbb{A}^1.$ Since $\mathcal{O}(X)$ is a mad subalgebra of $\mathcal{O}(T^*(X))$, then $\theta(\mathcal{O}(X))$
is a mad subalgebra in $\mathcal{O}(T^*(Y)).$ Hence $\theta(\mathcal{O}(X))=\mathcal{O}(Y)$ and we are done.

\end{proof}

\end{document}